\documentclass[a4,11pt]{article}
\usepackage{hyperref}
\usepackage{amsmath,amssymb,amsthm,mathrsfs,tikz,bbm,cleveref}
\usepackage[margin=2.5cm]{geometry}
\usepackage[textsize=tiny]{todonotes}

\newcommand{\Z}{\mathbb{Z}}

\renewcommand{\P}{\mathbb{P}}
\newcommand{\E}{\mathbb{E}}
\newcommand{\mc}[1]{\mathcal{#1}}

\newcommand{\bs}[1]{\boldsymbol{#1}}
\newcommand{\sss}[1]{\scriptscriptstyle #1}
\newcommand{\indic}[1]{\mathbbm{1}_{\{#1\}}}

\newcommand{\cweak}{\overset{w}{\to}}

\newcommand{\moo}{v}
\newtheorem{theorem}{Theorem}

\newtheorem{lemma}{Lemma}

\let\originalleft\left
\let\originalright\right
\renewcommand{\left}{\mathopen{}\mathclose\bgroup\originalleft}
\renewcommand{\right}{\aftergroup\egroup\originalright}

\theoremstyle{definition}

\newtheorem{rem}{Remark}

\newcommand{\floor}[1]{\lfloor #1 \rfloor}
\newcommand{\doubleblind}[1]{}
\title{The mean square displacement of random walk\\ on the Manhattan lattice}
\author{Nicholas R. Beaton\footnote{nrbeaton@unimelb.edu.au } \,  and Mark Holmes\footnote{holmes.m@unimelb.edu.au}\\
University of Melbourne\\
Parkville, VIC 3010\\
Australia }

\begin{document}
\maketitle
\abstract{We give an explicit formula for the mean square displacement of the random walk on the $d$-dimensional Manhattan lattice after $n$ steps, for all $n$ and all dimensions $d\ge 2$.}

\bigskip

\bigskip

Let $e_1,\dots, e_d$ be the standard basis vectors for $\Z^d$.   An \emph{oriented lattice} on $\Z^d$ is a directed graph on $\Z^d$ in which each bi-infinite line in $\Z^d$ has an orientation.  To be more precise, it is a directed graph with vertex set $\Z^d$ and directed edge set $E$ such that 
 for each $j\in [d]$ and each $x\in \Z^d$ with $x\cdot e_j=0$, exactly one of the following holds:
 \begin{itemize}
\item  for all $k\in \Z$, $ (x+ke_j,x+(k+1)e_j)\in E$ and $ (x+ke_j,x+(k-1)e_j)\notin E$, or
\item for all $k\in \Z$, $ (x+ke_j,x+(k-1)e_j)\in E$ and $ (x+ke_j,x+(k+1)e_j)\notin E$.
\end{itemize}
Given such a directed graph, one can ask about the behaviour of a random walk $\bs{X}=(X_n)_{n \in \Z_+}$ on the graph which chooses its next move uniformly from  the $d$ available directed edges at its current location.  At this level of generality the graph may be reducible in the sense that some sites might not be reachable from some other sites.  Two natural examples of irreducible  graphs are (i) the so-called Manhattan lattice (where orientations in neighbouring lines oscillate -- see below), and (ii) the setting where orientations of lines are determined by i.i.d.~fair coin tosses.  

The former (which is the topic of this paper) is much easier to understand than the latter.  For example, in 2 dimensions recurrence is know for (i) but is unresolved for (ii).  Moreover, the random walk in (i) is diffusive in all dimensions (e.g.~see below), while in (ii) it is believed to be diffusive only if $d\ge 4$ (see e.g.~\cite{KT,LTV,T}).
  For a modified 2-dimensional model (where both vertical steps are available from each site, but horizontal lines are oriented) introduced by Matheron and de Marsily%
 it has been shown that the i.i.d.~setting is in fact transient despite the oscillating case being recurrent \cite{CP,CG-PPS}.
  
 The main result of this paper is related to  the diffusivity of the random walk on the Manhattan lattice, which we now proceed to define explicitly.  The Manhattan lattice in $d\ge 2$ dimensions is the directed graph $M_d=(V,E)$ with vertex set $V=\Z^d=\{x=(x^{\sss[1]},\dots, x^{\sss[d]}): x^{\sss[i]}\in \Z \text{ for all }i \in [d]\}$ and (directed) edge set $E$ defined as follows:
\begin{align*}
(x,x+e_i)\in E & \text{ if and only if }\Big(\sum_{j \ne i}x^{\sss[j]}\Big)\text{ is even, and} \\
(x,x-e_i)\in E& \text{  if and only if }\Big(\sum_{j \ne i}x^{\sss[j]}\Big)\text{ is odd}.
\end{align*}
There are exactly $d$ directed edges pointing out of each $x\in V$, and $d$ directed edges pointing in.  For each $x\in \Z^d$ the sets $\{y:(x,y)\in E\}$ and $\{y: (y,x)\in E\}$ are disjoint and together include all $2d$ neighbours of $x$.   See Figure \ref{fig:manhat2} for a depiction of the case $d=2$ and Figure \ref{fig:manhat3} for a depiction of the case $d=3$.

\begin{figure}
\begin{center}
\begin{tikzpicture}[scale=0.6]
\foreach \x in {-6,-4,-2,0,2,4,6}
\foreach \y in {-6,-5,-4,-3,-2,-1,0,1,2,3,4,5,6}
{\draw[thick,->] (\x,\y)--(\x,\y+1/2);}

\foreach \x in {-5,-3,-1,1,3,5}
\foreach \y in {-6,-5,-4,-3,-2,-1,0,1,2,3,4,5,6}
{\draw[thick,->] (\x,\y)--(\x,\y-1/2);}

\foreach \x in {-6,-4,-2,0,2,4,6}
\foreach \y in {-6,-5,-4,-3,-2,-1,0,1,2,3,4,5,6}
{\draw[thick,->] (\y,\x)--(\y+1/2,\x);}

\foreach \x in {-5,-3,-1,1,3,5}
\foreach \y in {-6,-5,-4,-3,-2,-1,0,1,2,3,4,5,6}
{\draw[thick,->] (\y,\x)--(\y-1/2,\x);}
\node at (0,0) {$\bs{o}$};

\end{tikzpicture}
\end{center}
\caption{A portion of the Manhattan lattice in 2 dimensions.}
\label{fig:manhat2}
\end{figure}
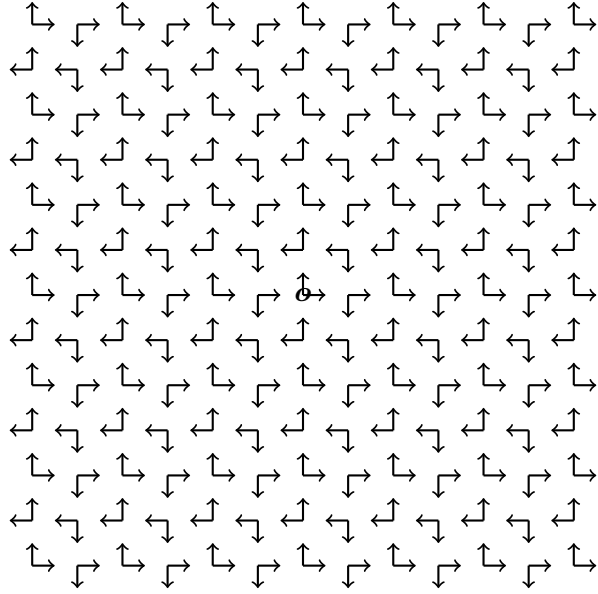
Fix $d$ and let $\bs{X}=(X_i)_{i \in \Z_+}$ be the Markov chain on  $M_d$ with $X_0=0\in \Z^d$ and transition probabilities $p_{x,x+e}=1/d$ if $(x,x+e)\in E$, and $0$ otherwise.   Then there are exactly $d^n$ distinct possibilities for the path $(X_0,\dots, X_n)$.  By observing the walker at all times $0=T_0<T_1<T_2<\dots$ when the walker is at sites in $A:=\{x\in \Z^d:x^{\sss{[i]}}\text{ is even for each }i\in [d]\}$ one can prove  recurrence in 2 dimensions and transience in larger dimensions, as well as a law of large numbers ($n^{-1}X_n \to 0$ a.s.) and central limit theorem $n^{-1/2}X_n\cweak \mc{N}(0,\Sigma)$ for some $\Sigma$ in general dimensions.  However it is not the case that $\E[X_n]=0$ ($\E[X_n]$ is to be understood as the vector $(\E[X_n^{\sss{[1]}}],\dots, \E[X_n^{\sss{[d]}}])$) nor that there exists $\sigma^2>0$ such that $\E[|X_n|^2]=\sigma^2n $ for all $n$.  Indeed we prove  the following result.
\begin{theorem}
\label{thm:main}
The location $X_n$ of the random walk in the $d\ge 2$-dimensional Manhattan lattice after $n\in \Z_+$ steps satisfies
\begin{align}
\E_d[X_n]&=\left(\sum_{i=1}^d e_i\right)\left[\dfrac{1-\big(\frac{2-d}{d}\big)^n}{2(d-1)}\right],\label{mean}\\
\E_d\Big[|X_n|^2\Big]&=\dfrac{\big(2(d-1)n-1\big)d^n+(2-d)^n}{d^{n-1}2(d-1)^2}.\label{main}
\end{align}
\end{theorem}
Our main result is \eqref{main}.  The numerator in \eqref{main} is always divisible by $2(d-1)^2$.  Note also that $n^{-1}\E_d[|X_n|^2] \to \frac{d}{d-1}$ as $n \to \infty$.  
\begin{rem}\label{rem:2d}
In the case $d=2$, \eqref{main} reduces to $\E_d\big[|X_n|^2\big]=2n-1$.
\end{rem}
Nadine Guillotin-Plantard  \cite{GP} in her PhD thesis proved that the random walk $\bs{X}$ is recurrent in 2 dimensions (and transient in higher dimensions) by examining return probabilities.  In 2~dimensions  $S_n=\floor{X_{2n}/2}$ is a simple symmetric random walk (where $\floor{\cdot}$ is applied componentwise) \cite{GP}.  Note that the number of distinct paths for the simple random walk (in 2 dimensions) of length $n$ is $4^n=2^{2n}$, where the latter is the number of distinct paths of length $2n$ in the Manhattan lattice.  Such a relation cannot hold in higher dimensions since we cannot have $(2d)^n=d^{mn}$ for any integer $m$ when $d>2$.  

As noted above, the graph $M_d$ is irreducible in general dimensions.  This fact is a special case of a more general result which is itself an exercise (see e.g.~\cite{HolmesRWRE}).   Since there are exactly $d$ directions available from each site, one can ask how many different distinct  local environments $\ell(x)=(\ell^{\sss[1]}(x),\dots,\ell^{\sss[d]}(x)) \in \{-1,1\}^d$ there are (as we vary $x$), where $\ell^{\sss[i]}(x) = 1$ if $(x,x+e_i) \in E$ and $\ell^{\sss[i]}(x)=-1$ otherwise.  Clearly there are at most $2^d$ possible local environments.   It is an exercise (see e.g.~\cite{HolmesRWRE}) to show that in even dimensions all of these local environments occur, while in odd dimensions exactly half of them occur.

\begin{figure}
\begin{center}
\begin{tikzpicture}
\foreach \x in {1,2,3,4}
\foreach \y in {1,2,3,4}
{
\node[circle,fill=black,scale=0.7] at (2*\x,2*\y) {};
\node[circle, fill=blue,scale=0.6] at (2*\x+0.6,2*\y+0.4) {};
\node[circle, fill=red,scale=0.5] at (2*\x+0.6+0.6,2*\y+0.4+0.4) {};
}

\foreach \x in {1,2,3,4}
\foreach \y in {1,2}
{
\draw[-stealth,very thick] (2*\x,4*\y)--(2*\x+1,4*\y);
\draw[-stealth,very thick,blue] (2*\x+0.6,4*\y+0.4)--(2*\x+0.6-1,4*\y+0.4);
\draw[-stealth,very thick,red,dotted] (2*\x+0.6+0.6,4*\y+0.4+0.4)--(2*\x+0.6+0.6+1,4*\y+0.4+0.4);
\draw[-stealth,very thick] (2*\x,4*\y-2)--(2*\x-1,4*\y-2);
\draw[-stealth,very thick,blue] (2*\x+0.6,4*\y+0.4-2)--(2*\x+0.6+1,4*\y+0.4-2);
\draw[-stealth,very thick,red,dotted] (2*\x+0.6+0.6,4*\y-2+0.8)--(2*\x+0.6+0.6-1,4*\y-2+0.4+0.4);
}

\foreach \y in {1,2,3,4}
\foreach \x in {1,2}
{
\draw[-stealth,very thick] (4*\x,2*\y)--(4*\x,2*\y+1);
\draw[-stealth,very thick] (4*\x-2,2*\y)--(4*\x-2,2*\y-1);
\draw[-stealth,very thick,blue] (4*\x+0.6,2*\y+0.4)--(4*\x+0.6,2*\y-1+0.4);
\draw[-stealth,very thick,blue] (4*\x+0.6-2,2*\y+0.4)--(4*\x+0.6-2,2*\y+1+0.4);
\draw[-stealth,very thick,red,dotted] (4*\x+0.6+0.6,2*\y+0.4+0.4)--(4*\x+0.6+0.6,2*\y+1+0.4+0.4);
\draw[-stealth,very thick,red,dotted] (4*\x+0.6+0.6-2,2*\y+0.4+0.4)--(4*\x+0.6+0.6-2,2*\y-1+0.4+0.4);
}

\foreach \y in {1,3}
\foreach \x in {1,3}{
\draw[-stealth,very thick] (2*\x,2*\y)--(2*\x+0.3,2*\y+0.2);
\draw[-stealth,very thick] (2*\x+2,2*\y)--(2*\x+2-0.3,2*\y-0.2);
\draw[-stealth,very thick,blue] (2*\x+0.6,2*\y+0.4)--(2*\x+0.6+0.3,2*\y+0.4+0.2);
\draw[-stealth,very thick,blue] (2*\x+0.6+2,2*\y+0.4)--(2*\x+0.6+2-0.3,2*\y-0.2+0.4);
\draw[-stealth,very thick,red,dotted] (2*\x+0.6+0.6,2*\y+0.4+0.4)--(2*\x+0.6+0.6+0.3,2*\y+0.4+0.4+0.2);
\draw[-stealth,very thick,red,dotted] (2*\x+0.6+0.6+2,2*\y+0.4+0.4)--(2*\x+0.6+0.6+2-0.3,2*\y-0.2+0.4+0.4);
}
\foreach \y in {2,4}
\foreach \x in {2,4}{
\draw[-stealth,very thick] (2*\x,2*\y)--(2*\x+0.3,2*\y+0.2);
\draw[-stealth,very thick] (2*\x-2,2*\y)--(2*\x-2-0.3,2*\y-0.2);
\draw[-stealth,very thick,blue] (2*\x+0.6,2*\y+0.4)--(2*\x+0.6+0.3,2*\y+0.2+0.4);
\draw[-stealth,very thick,blue] (2*\x+0.6-2,2*\y+0.4)--(2*\x+0.6-2-0.3,2*\y-0.2+0.4);
\draw[-stealth,very thick,red,dotted] (2*\x+0.6+0.6,2*\y+0.4+0.4)--(2*\x+0.6+0.6+0.3,2*\y+0.2+0.4+0.4);
\draw[-stealth,very thick,red,dotted] (2*\x+0.6+0.6-2,2*\y+0.4+0.4)--(2*\x+0.6+0.6-2-0.3,2*\y-0.2+0.4+0.4);
}
\end{tikzpicture}
\end{center}
\caption{A portion of the 3-dimensional Manhattan lattice:  vertices of the same colour are in the same plane, blue is ``behind'' black, and red is ``behind'' blue.}
\label{fig:manhat3}
\end{figure}
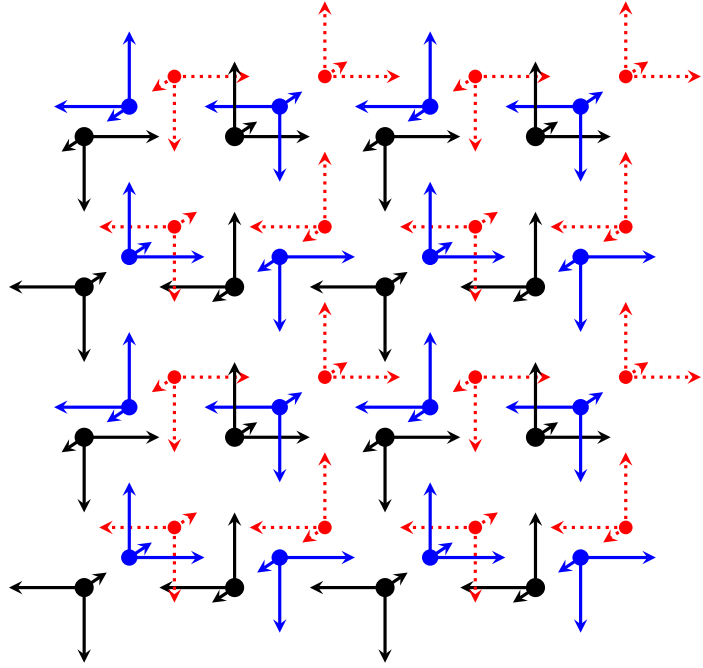

\begin{proof}[Proof of Theorem  \ref{thm:main} \eqref{mean}]
Note that a.s.
\begin{align*}
\E[X_{n+1}-X_n|X_n]=\frac{1}{d}\sum_{i=1}^de_i \Big[\indic{(\sum_{j\ne i}X_n^{\sss[j]}) \bmod 2=0}-\indic{(\sum_{j\ne i}X_n^{\sss[j]}) \bmod 2=1}\Big].
\end{align*}
Taking expectations of both sides yields
\begin{align}
\E[X_{n+1}-X_n]=\frac{1}{d}\sum_{i=1}^de_i \Big[2\P\Big((\sum_{j\ne i}X_n^{\sss[j]}) \bmod 2=0\Big)-1\Big].\label{giraffe1}
\end{align}
The probability appearing in \eqref{giraffe1} is just the probability that a Bin$(n,(d-1)/d)$ random variable is even.  Therefore 
\[\P\Big((\sum_{j\ne i}X_n^{\sss[j]}) \bmod 2=0\Big)=\frac12  +\frac12(1-2(d-1)/d)^n=\frac12+\frac12 \cdot \Big(\frac{2-d}{d}\Big)^n.\]
It follows that 
\[\E[X_{n+1}-X_n]=\frac{1}{d}\sum_{i=1}^de_i \Big(\frac{2-d}{d}\Big)^n.\]
Finally, we get that 
\[\E[X_n]=\sum_{j=0}^{n-1}\E[X_{j+1}-X_j]=\frac{1}{d}\sum_{i=1}^de_i \sum_{j=0}^{n-1}\Big(\frac{2-d}{d}\Big)^j,\]
and evaluating the geometric sum yields the claim.
\end{proof}

\begin{proof}[Proof of Theorem \ref{thm:main} \eqref{main}]
Let $\moo_n=\E_d[|X_n|^2]$.  Then $v_0=0$ and $v_1=1$.  We claim that for all $n\ge 0$,
\begin{equation}\label{eqn:vn_recurrence}
	\moo_{n+2}-\frac{2}{d}\moo_{n+1}-\Big(\frac{d-2}{d}\Big)\moo_n=2.
\end{equation}
It is then easy to verify that \eqref{main} solves this recurrence.  It therefore suffices to prove \eqref{eqn:vn_recurrence}.

Let $\mc{E}=\{e_1,\dots, e_d,-e_1,\dots, -e_d\}$.  Then for all $e \in \mc{E}$ we have 
 \begin{equation}
 |x+e|^2-|x|^2=2x \cdot e+1,\label{easysq}
 \end{equation}
since both sides are equal to $(x+e+x)\cdot (x+e-x)$.

Recall that $\ell(x)=(\ell^{\sss[1]}(x),\dots,\ell^{\sss[d]}(x)) \in \{-1,1\}^d$ denotes the local environment at vertex $x=(x^{[1]},\dots,x^{[d]})$, and let $L_n=\ell(X_n)$.
Using \eqref{easysq} with $e=X_{n+1}-X_n$ we get a.s.
	\begin{equation} 
		\mathbb{E}_d\left[|X_{n+1}|^2-|X_n|^2\,\big|\, X_n \right] = 2X_n \cdot \mathbb{E}_d\left[(X_{n+1}-X_n)\,\big|\,X_n\right]+1=	\frac{2}{d}X_n  \cdot  L_n+1.\label{hello1}
	\end{equation}
	Similarly, a.s.,
	\begin{align*} 
		\mathbb{E}_d\left[|X_{n+2}|^2-|X_{n+1}|^2\,\big|\, X_{n+1},X_n \right] &=	2X_{n+1}\cdot \E_d[ X_{n+2}-X_{n+1}|X_{n+1},X_n]+1\nonumber\\
		&=\frac{2}{d}X_{n+1}\cdot  L_{n+1}+1.
	\end{align*}	
	It follows that a.s.
	\[\mathbb{E}_d\left[|X_{n+2}|^2-|X_{n+1}|^2\,\big|\, X_n \right] =\frac{2}{d}\E_d\left[X_{n+1}\cdot L_{n+1}\big|\,X_n\right]+1.\]
Now \[L^{\sss[i]}_{n+1}=\begin{cases}
L^{\sss[i]}_{n}, & \text{ if }X_{n+1}-X_n=\pm e_i\\
-L^{\sss[i]}_{n}, & \text{ otherwise. }
\end{cases}\]
This means that $L_{n+1}=-L_n+2(X_{n+1}-X_n)$, so a.s.
\begin{align}
\mathbb{E}_d\left[|X_{n+2}|^2-|X_{n+1}|^2\,\big|\, X_n \right] &=\frac{2}{d}\E_d\left[X_{n+1}\cdot (-L_n+2(X_{n+1}-X_n))\,\big|\,X_n\right]+1\nonumber\\
&=\frac{2}{d}\Big[-L_n\cdot \E_d[X_{n+1}|X_n]+2\E_d[X_{n+1}\cdot(X_{n+1}-X_n)\,\big|\,X_n]\Big]+1\nonumber\\
&=\frac{2}{d}\Big[-L_n\cdot (\frac{1}{d}L_n+X_n)+2(1+X_n\cdot \E_d[X_{n+1}-X_n\,\big|\,X_n])\Big]+1\nonumber\\
&=\frac{2}{d}\Big[-1-L_n \cdot X_n+2(1+X_n\cdot \frac{1}{d}L_n)\Big]+1\nonumber\\
&=\frac{d+2}{d}+\frac{2(2-d)}{d^2}X_n \cdot L_n.\label{hello2}
\end{align}
Multiply \eqref{hello1} by $(d-2)/d$ and add to \eqref{hello2} to get, a.s., 
\begin{align*}
&\frac{d-2}{d}\mathbb{E}_d\left[|X_{n+1}|^2-|X_n|^2\,\big|\, X_n \right]+\mathbb{E}_d\left[|X_{n+2}|^2-|X_{n+1}|^2\,\big|\, X_n \right] \\
&=\frac{d-2}{d}\left[\frac{2}{d}X_n  \cdot  L_n+1\right]+\frac{d+2}{d}+\frac{2(2-d)}{d^2}X_n \cdot L_n\\
&=2.
\end{align*}
In other words, $v_{n+2}-\frac{2}{d}v_{n+1}-\frac{d-2}{d} v_n=2$, so \eqref{eqn:vn_recurrence} holds as claimed.
\end{proof}

\subsection*{Acknowledgements}
We thank Nadine Guillotin-Plantard for helpful conversations, as well as for  supplying us with an excerpt from her thesis.

\bibliographystyle{plain}

\begin{thebibliography}{99}

\bibitem{CP}
M. Campanino and D. Petritis.
\newblock Random walks on randomly oriented lattices.
\newblock \emph{ Mark. Proc. Relat. Fields} 9:391--412 (2003).

\bibitem{CG-PPS}
F.~Castell, N.~Guillotin-Plantard, F.~P\`ene, and B.~Schapira.
\newblock  A local limit theorem for random walks in random scenery
and on randomly oriented lattices. 
\newblock \emph{Ann. Probab.} 39:2079--2118, (2011).


\bibitem{GP}
N.~Guillotin.  
\newblock Marche al\'eatoire dynamique dans une sc\`ene al\'eatoire. Probl\`emes li\'es \`a l'inhomog\'en\'eit\'e spatiale de certaines cha\^ines de Markov. 
\newblock \emph{PhD thesis, Universit\'e de Rennes 1} (1997).


\bibitem{HolmesRWRE}
M.~Holmes.
\newblock A first course on random walks in random environments. \emph{Unpublished notes} (2022).

\bibitem{KT}
G.~Kozma and B.~T\'oth. 
\newblock Central limit theorem for random walks in doubly stochastic random
environment: $\mathscr{H}_{-1}$ suffices. 
\newblock \emph{Ann. Prob.}, 45:4307--4347, (2017). MR-3737912.

\bibitem{LTV}
S.~Ledger. B.~T\'oth, and  B.~Valk\'o. 
\newblock Random walk on the randomly-oriented Manhattan lattice.
\newblock \emph{ Electron. Commun. Probab.} 23:1--11, (2018). https://doi.org/10.1214/18-ECP144

%
%
%
%
%


\bibitem{T}
B.~T\'oth. 
\newblock Quenched central limit theorem for random walks in doubly stochastic random environment.
\newblock \emph{Ann. Probab}. 46 (6) 3558--3577, (2018).


\end{thebibliography}

\end{document}